\documentclass[11pt, oneside]{article}
\usepackage{amsfonts,amssymb,amsmath,latexsym,amsthm}
\usepackage{mathrsfs}
\usepackage{color}
\usepackage[colorlinks]{hyperref}
\usepackage{enumerate,indentfirst,mathtools,verbatim,authblk}
\usepackage{tikz}
\usepackage{psfrag}
\usepackage{graphicx}
\usepackage{bm}
\usepackage[rflt]{floatflt}
\usepackage{float}
\usepackage[square, comma, sort&compress, numbers]{natbib}
\numberwithin{equation}{section}
\allowdisplaybreaks

\newcommand{\upcite}[1]{\textsuperscript{\textsuperscript{\cite{#1}}}}

\newtheorem{theorem}{Theorem}[section]

\newtheorem{lemma}[theorem]{Lemma}

\theoremstyle{definition}

\theoremstyle{remark}
\newtheorem{remark}[theorem]{Remark}
\numberwithin{equation}{section}

\newenvironment{proof3.1}{\medskip\noindent{\bf Proof of Theorem \ref{thm3.1}:}\enspace}{\hfill \qed  \medskip}
\newenvironment{proof2.11}{\medskip\noindent{\bf Proof of Theorem \ref{thm2.10}:}\enspace}{\hfill \qed  \medskip}

\newenvironment{proof2.13}{\medskip\noindent{\bf Proof of Theorem \ref{thm2.11}:}\enspace}{\hfill \qed  \medskip}
\newenvironment{proof2.15}{\medskip\noindent{\bf Proof of Theorem \ref{thm2.15}:}\enspace}{\hfill \qed  \medskip}

\usepackage{amssymb,amsmath}
\begin{document}
\title{\LARGE\bf {Energy decay rates of  solutions to a viscoelastic wave equation with variable exponents and weak damping}
\thanks{The project is supported by NSFC(11301211),
 by the Scientific and Technological Project of Jilin Province's Education Department in Thirteenth-five-year(JJKH20180111KJ).}}
\author[1]{Menglan Liao}
\author[2]{Bin Guo\thanks{Corresponding author
\newline \hspace*{4mm}{
Email addresses: liaoml14@mails.jlu.edu.cn(M. Liao), bguo@jlu.edu.cn(B. Guo), zxyzhu2009@163.com.cn(X. Zhu)}}}
\author[2,3]{Xiangyu Zhu}
\affil[1]{School of Mathematical Sciences, Xiamen University, Xiamen, Fujian, 361005, China}
\affil[2]{School of Mathematics, Jilin University, Changchun, Jilin Province 130012, China}
\affil[3]{Department of Mathematics, College of Science, Yanbian University,

Yanji, Jilin Province  133002, China}
\renewcommand*{\Affilfont}{\small\it}
\date{} \maketitle
\vspace{-20pt}
{\bf Abstract:}
The goal of the present paper is to study the asymptotic behavior of solutions for the viscoelastic wave equation with variable exponents
\[ u_{tt}-\Delta u+\int_0^tg(t-s)\Delta u(s)ds+a|u_t|^{m(x)-2}u_t=b|u|^{p(x)-2}u\]
under initial-boundary condition, where the exponents $p(x)$ and $m(x)$ are given functions, and $a,~b>0$ are constants. More precisely, under the condition $g'(t)\le -\xi(t)g(t)$,
here $\xi(t):\mathbb{R}^+\to\mathbb{R}^+$ is a non-increasing differential function with $\xi(0)>0,~\int_0^\infty\xi(s)ds=+\infty$, general decay results are derived. In addition, when $g$ decays polynomially, the exponential and  polynomial decay rates are obtained as well, respectively. This work generalizes and improves earlier results in the literature.



\maketitle

\section{Introduction}
In this paper, we are concerned with the energy decay estimates to the following viscoelastic wave equation with variable exponents:
\begin{equation}
\label{1.1}
\begin{cases}
      u_{tt}-\Delta u+\int_0^tg(t-s)\Delta u(s)ds+a|u_t|^{m(x)-2}u_t=b|u|^{p(x)-2}u& \text{in}~\Omega \times(0,T),  \\
      u(x,t)=0&\text{on}~\partial\Omega \times(0,T),  \\
      u(x,0)=u_0(x),~u_t(x,0)=u_1(x)& \text{for}~x\in\Omega,
\end{cases}
\end{equation}
where $T>0$, $\Omega\subset \mathbb{R}^n~(n\geq3)$ is a bounded Lipschitz domain with sufficiently smooth boundary $\partial\Omega$, and $a,~b>0$ are constants. The exponents $m(x)$ and $p(x)$ are given measurable functions on $\Omega$ satisfying
\begin{equation}\label{1.2}
    2\leq q_1\leq q(x)\leq q_2\leq \frac{2n}{n-2},
\end{equation}
with $q_1:=ess\inf_{x\in\Omega}q(x),~q_2:=ess\sup_{x\in\Omega}q(x)$, and the log-H\"{o}lder continuity condition, i.e. for $A>0$ and $0<\delta<1$:
\begin{equation}\label{1.3}
   |q(x)-q(y)|\leq-\frac{A}{\log|x-y|},\quad \text{for a.e. }x,y\in \Omega,~\text{with}~|x-y|<\delta.
\end{equation}
Let $g:\mathbb{R}^+\to\mathbb{R}^+$ be a non-increasing and differential function satisfying
\begin{equation}\label{2.1}
    g(0)>0,\quad 1-\int_0^\infty g(s)ds:=l>0.
\end{equation}

Viscoelasticity is the property of materials that exhibit both viscous and elastic characteristics when
undergoing deformation. The study of viscoelastic problems has attracted the attention of many authors and several decay results have been established. Cavalcanti and Oquendo \cite{CO2003} studied the equation
\begin{equation}\label{01}
    u_{tt}-\kappa_0\Delta u+\int_0^t\mathrm{div}[a(x)g(t-s)\nabla u(s)]ds+f(u)+b(x)h(u_t)=0\quad\text{in}~\Omega \times\mathbb{R}^+,
\end{equation}
where $a,~b$ are nonnegative functions, $a\in C^1(\overline{\Omega})$, $b\in L^\infty(\Omega)$ satisfying $a(x)+b(x)\geq\delta>0$ for any $ x\in\Omega,$ and $f$, $h$ are power-like functions.  They proved exponential stability for $g$ decaying exponentially and $h$ linear as well as polynomial stability for $g$ decaying polynomially and $h$ nonlinear under suitable conditions, which improved a result obtained by Cavalcanti et al. \cite{CCS2002}, where an exponential rate of decay was revealed by assuming that the kernel $g$ in the memory term decays exponentially for $a(x)=1,~\kappa_0=1$ in equation \eqref{01}. For $f(u)=0,~b(x)h(u_t)=0$ in equation \eqref{01}, many results on decay estimates of solutions have been presented, the interested readers can refer to \cite{CCM2008,M2008,M2018,RS1997,RS2001,T2013}, and the references therein.  It is worth pointing out that $g$ is not limited to the exponential decay or polynomial decay. For instance, the assumption on $g$ in \cite{M2008} is that there exists a differentiable function $\xi$ satisfying
\begin{equation*}
    \begin{split}
    g'(t)\leq -\xi(t)g(t)\quad \forall~t\geq 0,\\
    \Big|\frac{\xi'(t)}{\xi(t)}\Big|\leq k,~\xi(t)>0,~\xi'(t)\leq 0\quad \forall~t>0.
    \end{split}
\end{equation*}

Liu \cite{L2010} investigated the nonlinear viscoelastic equation
\begin{equation}\label{02}
    u_{tt}-\Delta u+\int_0^tg(t-s)\nabla u(s)ds+a(x)|u_t|^mu_t+|u|^\gamma u=0\quad\text{in}~\Omega \times\mathbb{R}^+,
\end{equation}
where $\gamma>0,~m\geq 0$, $a(x):\Omega\to \mathbb{R}^+$ is a function, which may vanish on
any part of $\Omega$ (including $\Omega$ itself), he established exponential or polynomial decay result if the relaxation function $g$ decays polynomially. Their results improved an earlier one given by Berrimi and Messaoudi \cite{BM2004}, where  an exponential decay result was presented when $g$ decays exponentially without imposing geometry restrictions on the boundary $\partial\Omega$. Recently, Belhannache, Algharabli and Messaoudi \cite{BAM2020} studied equation \eqref{02} in the absence of $|u|^\gamma u$ and $a(x)\equiv1$, they proved an explicit and general decay rate results by assuming that there exists a $C^1$ function $G: \mathbb{R}^+\to \mathbb{R}^+$ which is linear or it is strictly increasing and strictly convex $C^2$ function on $(0, r],~r\leq g(0)$, with $G(0)=G'(0)=0$, such that
\[g(t)\leq-\xi(t)G(g(t))\quad \forall~t\geq 0,\]
where $\xi(t)$ is a positive non-increasing differentiable function.

When there exist both the frictional damping and source term in equation, i.e.
\begin{equation}\label{03}
     u_{tt}-\Delta u+\int_0^tg(t-s)\Delta u(s)ds+a|u_t|^{m-2}u_t=b|u|^{p-2}u\quad\text{in}~\Omega \times\mathbb{R}^+,
\end{equation}
the global existence for $p\leq m$ and blow-up results for $p>m$ of solutions have been obtained in \cite{M2003,M2006,S2015}. For $m\equiv2$ in equation \eqref{03}, Wang et al. \cite{WW2008} investigated explicit exponential energy decay of the viscoelastic wave equation under the potential well. For $m>2$, however, there was hardly any work concerning the decay estimates of solutions. Until recently, Guo,  Rammaha and Sakuntasathien \cite{GRS2018} considered the asymptotic behavior of solutions for the history value problem
\begin{equation*}
\begin{cases}
    u_{tt}-k(0)\Delta u-\int_0^\infty k'(s)\Delta u(t-s)ds+a|u_t|^{m-1}u_t=|u|^{p-1}u\quad&\text{in}~\Omega \times(0,T),\\
    u(x,t)=0\quad&\text{in}~\Gamma \times(-\infty,T),\\
    u(x,t)=u_0(x,t)\quad&\text{in}~\Omega \times(-\infty,0].
\end{cases}
\end{equation*}
They presented global existence of solutions by introducing a suitable
notion of a potential well provided that the history value $u_0$ is taken from a subset of the potential well. Also, the uniform energy decay rate depending on the behavior of the damping term near the origin as well as the decreasing rate of the relaxation kernel $-k'(s)$. To be more precise, if the damping is linear near the origin and the relaxation kernel decays to zero exponentially, then the energy also decays to zero exponentially fast. Otherwise, the energy decays polynomially. In particular, Hadamard well-posedness of problem above has been discussed in \cite{GRS2014,GRS2017,SLW2018}.

In the absence of the viscoelastic term $\int_0^tg(t-s)\Delta u(s)ds$, there exist only several papers concerning the decay estimate of solutions for the hyperbolic equations with variable-exponent nonlinearities, among them were obtained by  Messaoudi,  Al-Smail and  Talahmeh \cite{MAT2018}, Ghegal,  Hamchi and Messaoudi \cite{GHM2018},  Messaoudi \cite{M2020} and Li, Guo and Liao \cite{LGL2020}, respectively.  Actually,  Park and Kang \cite{PK2019} have established the local existence and  the blow-up result of solutions for problem \eqref{1.1} when the initial energy lies in positive as well as nonpositive. However, there dose not exist any work related to the asymptotic behavior of solutions for problem \eqref{1.1}. In fact, there is no result concerning the asymptotic stability of the viscoelastic wave equation with variable exponent. In this paper, we follow the part idea in our previous paper \cite{LGL2020} to study the asymptotic behavior of solutions. 

This paper is organized as follows. In Section 2, we recall some useful lemmas and known results. Energy decay rates will be presented in Section 3.

\section{Preliminaries}
In this section, we present some useful results. Throughout this paper, we denote by $\|\cdot\|_q$ the $L^q(\Omega)$ with $1\leq q<\infty$. 
And we denote $u(x,t)$ by $u(t)$. In order to discuss problem (\ref{1.1}), we begin with  recalling Banach spaces of Orlicz-Sobolev type $L^{q(x)}(\Omega)$ \cite{FZ2001,FZ2003},
\[
L^{q(x)}(\Omega)=\Big\{f\Big| f~\mathrm{is~a~measurable~real-valued~function},\int_{\Omega}{|f|^{q(x)}}dx<\infty\Big\}.
\]
We introduce the norm on $L^{q(x)}(\Omega)$ by
\[\|f\|_{q(x)}=\mathrm{inf}\Big\{\lambda>0\Big|\int_{\Omega}{\Big|\frac{f}{\lambda}\Big|^{q(x)}}dx\leq1\Big\}.\]
It follows directly that
\begin{equation}\label{add2.1}
    \min\left\{\|f\|_{q(x)}^{q_1},\|f\|_{q(x)}^{q_2}\right\}\leq \int_{\Omega}{|f|^{q(x)}}dx\leq\max\left\{\|f\|_{q(x)}^{q_1},\|f\|_{q(x)}^{q_2}\right\}.
\end{equation}
\begin{lemma}\label{lem2.1}
\upcite{FZ2001,FZ2003} If $q_1(x),~q_2(x)\in C_+(\overline\Omega)=\{h\in C(\overline\Omega):~\min\limits_{x\in \overline\Omega}h(x)>1\},$ $q_1(x)\leq q_2(x)$ for any $x\in \Omega$, then there exists the continuous embedding $L^{q_2(x)}(\Omega)\hookrightarrow L^{q_1(x)}(\Omega),$ whose norm does not exceed $|\Omega|+1$.
\end{lemma}

Now, we directly give the local existence and blow-up results of solutions for problem \eqref{1.1}, which was proved in a recent paper published by Park and Kang \cite{PK2019}.
\begin{theorem}
\upcite{PK2019}
Suppose that \eqref{2.1} hold, $m(x),~p(x)$ satisfy \eqref{1.2} \eqref{1.3} and
\[2<p_1\leq p(x)\leq p_2<\frac{2(n-1)}{n-2}.\]
Then$,$ for every $(u_0(x),u_1(x))\in H_0^1(\Omega)\times L^2(\Omega),$ problem \eqref{1.1} has a unique local solution for some $T>0,$
\[u\in C([0,T];H_0^1(\Omega)),\quad u_t\in C([0,T];L^2(\Omega))\cap L^{m(x)}(\Omega\times(0,T)). \]
\end{theorem}
Define the total energy associated with problem \eqref{1.1} by
\begin{equation}\label{2.2}
\begin{split}
    E(t)&=\frac12\|u_t(t)\|_2^2+\frac12\Big(1-\int_0^t g(s)ds\Big)\|\nabla u(t)\|_2^2\\
    &\quad+\frac12\int_0^tg(t-s)\|\nabla u(t)-\nabla u(s)\|_2^2ds-b\int_\Omega\frac{|u(t)|^{p(x)}}{p(x)}dx.
\end{split}
\end{equation}
Then, for $t\geq 0$
\begin{equation}\label{2.3}
\begin{split}
    E'(t)=-a\int_\Omega |u_t(t)|^{m(x)}dx-\frac12g(t)\|\nabla u(t)\|_2^2+\frac12\int_0^tg_t(t-s)\|\nabla u(t)-\nabla u(s)\|_2^2ds\leq 0.
\end{split}
\end{equation}
Throughout this paper, we denote by $T_{max}$ the maximal existence time, and set
\[B_1=\max\Big\{1,\frac{B}{l^{\frac12}},\frac{1}{b^{\frac12}}\Big\},
\quad\lambda_1=\Big(\frac{1}{bB_1^{p_1}}\Big)^{\frac{1}{p_1-2}},\quad E_1=\Big(\frac12-\frac{1}{p_1}\Big)\lambda_1^2,\quad \lambda(0)=l^{\frac12}\|\nabla u_0\|_2\]
with $B$ being the embedding constant of $H_0^1(\Omega)\hookrightarrow L^{q(x)}(\Omega)$, i.e.
\begin{equation}\label{add2}
    \|u\|_{q(x)}\leq B\|\nabla u\|_2.
\end{equation}
\begin{theorem}
\upcite{PK2019} Suppose that $m(x),~p(x)$ satisfy \eqref{1.2} \eqref{1.3} and $m_2<p_1$. Assume that $g$ satisfies \eqref{2.1} and
\[\int_0^\infty g(s)ds<\frac{p_1/2-1}{p_1/2-1+1/(2p_1)}.\]
Then the solution of problem \eqref{1.1} blows up in finite time if
\[E(0)<\Big(1-\frac{1-l}{p_1(p_1-2)l}\Big)E_1,\quad \lambda_1<\lambda(0).\]
\end{theorem}

\section{Energy decay rates}
In  this section, the energy decay rate of global solutions for problem \eqref{1.1} is discussed. Our result is as follows:
\begin{theorem}\label{thm3.1}
Suppose that $m(x),~p(x)$ satisfy \eqref{1.2} \eqref{1.3}. If the relaxation function $g$ satisfies \eqref{2.1} and the following conditions are fulfilled
\begin{equation}\label{3.1}
    0<E(0)<\Big(\frac{p_1}{p_2}\Big)^{\frac{1}{p_1-2}}\lambda_1^2\Big(\frac12-\frac{1}{p_2}\Big),\quad \lambda_1>\lambda(0),
\end{equation}
 then there admits the following exponential or polynomial decay rate$:$
\begin{enumerate}[$(1)$]
  \item if  $g'(t)\leq -\xi(t)g(t)$ and $\xi(t):\mathbb{R}^+\to\mathbb{R}^+$ is a non-increasing differential function with $\xi(0)>0,~\int_0^\infty \xi(s)ds=+\infty$, then
   \begin{equation}\label{00add0}
    E(t)\leq E(0)\left(\frac{m_2}{2+K(m_2-2)\int_0^t\xi(s)ds}\right)^{\frac{2}{m_2-2}}\quad\text{for~}m_2>2,
  \end{equation}
  \begin{equation}\label{add0}
    E(t)\leq E(0)e^{1-K\int_0^t\xi(s)ds} \quad\text{for~}m(x)\equiv2,
  \end{equation}
  where $K>0$ is obtained later.
  \item if $g'(t)+Cg^\alpha(t)\leq 0$ with $C$ being some positive constant,  then
   \begin{equation}\label{00add1}
     E(t)\leq E(0)\left(\frac{m_2}{2+K(\alpha,\sigma)(m_2-2)t}\right)^{\frac{2}{m_2-2}}\quad \text{for~}m_2>2,
  \end{equation}
  \begin{equation}\label{add1}
    E(t)\leq E(0)e^{1-K(\alpha,\sigma)t}\quad \text{for~}m(x)\equiv2,
  \end{equation}
   where $K(\alpha,\sigma)>0$  is obtained later$,$  $1<\alpha<2,~0<\sigma<1$ and $2\alpha+\sigma<3.$
\end{enumerate}
And the total energy $E(t)$ decays to zero$,$ i.e.
\[\lim_{t\to \infty}E(t)=0.\]
\end{theorem}

\begin{remark}
If $p(x)$ is constant, i.e. $p(x)\equiv p$, obviously, for $0<E(0)<E_1$, we have the corresponding  exponential or polynomial decay rate.  If $p(x)$ is variable, however, we do not discuss if there exists the energy decay  result of global solutions  for $\Big(\frac{p_1}{p_2}\Big)^{\frac{1}{p_1-2}}\lambda_1^2\Big(\frac12-\frac{1}{p_2}\Big)\leq E(0)$.
\end{remark}

 In order to prove our theorem, let us show some auxiliary lemmas.
 \begin{lemma}\label{lem3.0}
\upcite{M1999}
Let $E: \mathbb{R}^+\to\mathbb{R}^+$ be a non-increasing function and $\phi : \mathbb{R}^+\to\mathbb{R}^+$ be a strictly increasing function of class $C^1$ such that
\[\phi(0)=0\text{ and }\phi(t)\to+\infty\text{ as }t\to+\infty.\]
Assume that there exist $\sigma \geq  0$ and $\omega > 0$ such that$:$
 \[\int_t^{+\infty}(E(s))^{1+\sigma}\phi^{\prime}(s)ds\leq  \frac{1}{\omega}(E(0))^\sigma E(t),\]
 then $E$ has the following decay property$:$
\begin{enumerate}[$(1)$]
  \item if $\sigma=0,$ then $E(t)\leq  E(0)e^{1-\omega\phi(t)}$ for all $t\geq  0;$
  \item if $\sigma>0,$ then $E(t)\leq  E(0)\left(\frac{1+\sigma}{1+\omega\sigma\phi(t)}\right)^{\frac{1}{\sigma}}$ for all $t\geq  0$.
\end{enumerate}
\end{lemma}

 \begin{lemma}\label{lem3.1}
 Suppose that $m(x),~p(x)$ satisfy \eqref{1.2} \eqref{1.3}. Assume further that $g$ satisfies \eqref{2.1} and
\[E(0)<E_1,\quad \lambda_1>\lambda(0),\]
there exists a constant $0<\lambda_2<\lambda_1$ such that
\[l\|\nabla u(t)\|_2^2\leq \lambda_2^2\quad \forall~t\in [0,T_{max}).\]
\end{lemma}
\begin{proof}
It follows from \eqref{add2} that
\begin{equation}\label{3.2}
\begin{split}
    E(t)&\geq\frac12\Big(1-\int_0^t g(s)ds\Big)\|\nabla u(t)\|_2^2-b\int_\Omega\frac{|u(t)|^{p(x)}}{p(x)}dx\\
    &\geq\frac{l}{2}\|\nabla u(t)\|_2^2-\frac{b}{p_1}\max\Big\{\|u(t)\|_{p(x)}^{p_1},\|u(t)\|_{p(x)}^{p_2}\Big\}\\
    &\geq\frac{l}{2}\|\nabla u(t)\|_2^2-\frac{b}{p_1}\max\Big\{(B\|\nabla u(t)\|_2)^{p_1},(B\|\nabla u(t)\|_2)^{p_1}\Big\}\\
    &\geq\frac{1}{2}\lambda^2(t)-\frac{b}{p_1}B_1^{p_1}\max\Big\{\lambda^{p_1}(t),\lambda^{p_2}(t)\Big\}
    :=f(\lambda(t))
\end{split}
\end{equation}
with $\lambda(t)=l^{\frac12}\|\nabla u(t)\|_2.$ Clearly, $f(\lambda(t))$ satisfies the following properties:
\begin{align*}
f'(\lambda(t))&= \begin{cases}
\lambda(t)-\frac{bB_1^{p_1}p_2}{p_1}\lambda^{p_2-1}(t)<0&\quad\text{for~}\lambda(t)>1,\\
\lambda(t)-bB_1^{p_1}\lambda^{p_1-1}(t)&\quad\text{for~}0<\lambda(t)<1;
\end{cases} \\
f'_+(1)&=1-\frac{bB_1^{p_1}p_2}{p_1}<0,\quad f'_-(1) =1-bB_1^{p_1}<0;\\
f'(\lambda_1)&=0,\quad0<\lambda_1<1.
\end{align*}
It is easily verified that $f(\lambda(t))$ is strictly increasing for $0<\lambda(t)<\lambda_1$ and $f(\lambda(t))$ is strictly decreasing for $\lambda_1<\lambda(t)$, $f(\lambda(t))\rightarrow -\infty$ as  $\lambda\rightarrow +\infty$, and $f(\lambda_1)=E_1$.

Since $0<E(0)<E_1$, there exists a positive constant $\lambda_{2}<\lambda_{1}$ such that $f(\lambda_{2})=E(0)$. Recalling $\lambda(0)=l^{\frac12}\|\nabla u_0\|_2$, then we have $f(\lambda(0))\leq E(0)=f(\lambda_{2})$ by \eqref{3.2}, which implies that $\lambda(0)\leq\lambda_{2}$. In order to prove $l\|\nabla u(t)\|_2^2\leq \lambda_2^2$ for any $t\in(0,T_{max})$, we suppose on the contrary that $l\|\nabla u(t_0)\|_2^2> \lambda_2^2$ for some $t_{0}\in(0,T_{max})$ since $\lambda(0)\leq\lambda_{2}<\lambda_{1}$. The continuity of $l\|\nabla u(t)\|_2^2$ and $\lambda(0)<\lambda_{1}$ implies that we can choose $0<t^{*}<t_{0}$ such that $\lambda_{2}<l^{\frac12}\|\nabla u(t^{*})\|_2<\lambda_{1}.$ Therefore, it follows from \eqref{3.2} that
\[E(0)=f(\lambda_{2})<f(l^{\frac12}\|\nabla u(t^{*})\|_2)\leq E(t^{*}),\]
 which contradicts \eqref{2.3}.
\end{proof}

For the sake of simplicity, define the quadratic energy $\mathscr{E}(t)$ by
\[\mathscr{E}(t)=\frac12\|u_t(t)\|_2^2+\frac12\Big(1-\int_0^t g(s)ds\Big)\|\nabla u(t)\|_2^2+\frac12\int_0^tg(t-s)\|\nabla u(t)-\nabla u(s)\|_2^2ds,\]
 it follows from \eqref{2.2} that
\begin{equation}\label{add3.1}
    E(t)=\mathscr{E}(t)-b\int_\Omega\frac{|u(t)|^{p(x)}}{p(x)}dx.
\end{equation}

\begin{lemma}\label{lem3.2}
Let all assumptions of Lemma $\ref{lem3.1}$ hold, then for any $t\in[0, T_{max})$,
\begin{equation}\label{3.3}
    b\int_\Omega\frac{|u(t)|^{p(x)}}{p(x)}dx\leq \tilde{C}E(t)
    \leq \tilde{C}E(0);
\end{equation}
\begin{equation}\label{3.5}
     \mathscr{E}(t)\leq (1+\tilde{C})E(t)\leq (1+\tilde{C})E(0)
\end{equation}
with \[\tilde{C}= \frac{\frac{2bB_1^{p_2}}{p_1}\lambda_2^{p_1-2}}{1-\frac{2bB_1^{p_2}}{p_1}\lambda_2^{p_1-2}}.\]
\end{lemma}
\begin{proof}
It directly follows from \eqref{3.2} \eqref{2.2} and Lemma $\ref{lem3.1}$ that
\begin{equation*}
\begin{split}
b\int_\Omega\frac{|u(t)|^{p(x)}}{p(x)}dx
&\leq\frac{b}{p_1}\max\Big\{\|u(t)\|_{p(x)}^{p_1},\|u(t)\|_{p(x)}^{p_2}\Big\}\\
&\leq\frac{b}{p_1}\max\Big\{\Big(B\|\nabla u(t)\|_2\Big)^{p_1},\Big(B\|\nabla u(t)\|_2\Big)^{p_2}\Big\}\\
&\le\frac{bB_1^{p_2}}{p_1}\max\Big\{\Big(l^{\frac12}\|\nabla u(t)\|_2\Big)^{p_1-2},
\Big(l^{\frac12}\|\nabla u(t)\|_2\Big)^{p_2-2}\Big\}l\|\nabla u(t)\|_2^2\\
&\leq\frac{2bB_1^{p_2}}{p_1}\max\Big\{\lambda_2^{p_1-2},
\lambda_2^{p_2-2}\Big\}\Big(E(t)+b\int_\Omega\frac{|u(t)|^{p(x)}}{p(x)}dx\Big)\\
&=\frac{2bB_1^{p_2}}{p_1}\lambda_2^{p_1-2}\Big(E(t)+b\int_\Omega\frac{|u(t)|^{p(x)}}{p(x)}dx\Big),
\end{split}
\end{equation*}
which implies \eqref{3.3} by using \eqref{2.3}.
Combining \eqref{3.3} and \eqref{add3.1}, we get
\begin{equation*}
  \mathscr{E}(t)= E(t)+b\int_\Omega\frac{|u(t)|^{p(x)}}{p(x)}dx\leq (1+\tilde{C})E(t)\leq (1+\tilde{C})E(0).
\end{equation*}
\end{proof}

\begin{remark}
 Clearly, \eqref{3.5} indicates that $\mathscr{E}(t)$ is uniformly bounded for all $t\in[0, T_{max})$. This implies that the solution is global and the maximal existence time $T_{max} =\infty$. And we have
 \begin{equation}\label{add3.2}
    0\leq E(t)\leq E(0)\quad \text{for all }t\in[0,\infty).
 \end{equation}
\end{remark}

In what follows, let us prove our theorem based on lemmas above.

\begin{proof3.1}
Multiplying the first equality of problem \eqref{1.1} by $\xi(t)E^{\gamma}(t)u(t)$ and integrating over $\Omega\times(s,T)$ with $\gamma$ being obtained later yield
\begin{equation}\label{3.6}
\begin{split}
    &\int_s^T\xi(t)E^{\gamma}(t)\int_\Omega u_{tt}(t)u(t)dxdt+\int_s^T\xi(t)E^{\gamma}(t)\|\nabla u(t)\|_2^2dt\\
    &\quad\quad+\int_s^T\xi(t)E^{\gamma}(t)\int_\Omega\int_0^t g(t-s)\Delta u(s)dsu(t)dxdt\\
    &\quad\quad+a\int_s^T\xi(t)E^{\gamma}(t)\int_\Omega|u_t(t)|^{m(x)-2}u_t(t)u(t)dxdt\\
    &\quad=b\int_s^T\xi(t)E^{\gamma}(t)\int_\Omega|u(t)|^{p(x)}dxdt.
\end{split}
\end{equation}
Integrating by parts implies that 
\begin{equation}\label{3.7}
    \begin{split}
       &\int_s^T\int_\Omega\int_0^t g(t-s)\Delta u(s)ds u(t)dx dt =-\int_s^T\int_\Omega\int_0^t g(t-s)\nabla u(s)\nabla u(t) ds dxdt \\
         &\quad=-\int_s^T\int_\Omega\int_0^t g(t-s)\nabla  u(t)\nabla (u(s)-u(t))ds dxdt-\int_s^T\int_0^t g(s)ds\|\nabla u(t)\|_2^2dt.
     \end{split}
\end{equation}
Inserting \eqref{3.7} into \eqref{3.6}, one obtains

\begin{equation}\label{3.8}
  \begin{split}
    &\int_s^T \xi(t)E^\gamma(t)\frac{d}{dt}\int_\Omega u(t)u_t(t)dxdt-\int_{s}^{T}\xi(t)E^{\gamma}(t)\|u_{t}(t)\|_2^{2}dt\\
    &\quad\quad+\int_s^T\xi(t)E^\gamma(t)\Big(1-\int_0^t g(s)ds\Big)\|\nabla u(t)\|_2^2dt\\
    &\quad\quad-\int_s^T\xi(t)E^\gamma(t)\int_\Omega\int_0^t g(t-s)\nabla  u(t)\nabla (u(s)-u(t))ds dxdt\\
    &\quad\quad+a\int_s^T\xi(t)E^\gamma(t)\int_\Omega|u_t(t)|^{m(x)-2}u_t(t)u(t)dxdt\\
    &\quad=b\int_s^T\xi(t)E^\gamma(t)\int_\Omega|u(t)|^{p(x)}dxdt.
\end{split}
\end{equation}
The definition of $E(t)$ in \eqref{2.2} and \eqref{3.8} indicate that
\begin{equation}\label{3.9}
    \begin{split}
    &2\int_s^T \xi(t)E^{\gamma+1}(t)dt=-\int_{s}^{T}\frac{d}{d t}\Big[\xi(t)E^{\gamma}(t)\int_{\Omega}uu_{t}dx\Big]dt\\
    &\quad+
\gamma\int_{s}^{T}\xi(t)E^{\gamma-1}(t)E'(t)\int_{\Omega}uu_{t}dxdt+
\gamma\int_{s}^{T}\xi'(t)E^{\gamma}(t)\int_{\Omega}uu_{t}dxdt\\
&\quad+2\int_{s}^{T}\xi(t)E^{\gamma}(t)\|u_{t}(t)\|_2^{2}dt+\int_s^T\xi(t)E^\gamma(t)\int_0^tg(t-s)\|\nabla u(t)-\nabla u(s)\|_2^2dsdt\\
    &\quad+\int_s^T\xi(t)E^\gamma(t)\int_\Omega\int_0^t g(t-s)\nabla  u(t)\nabla (u(s)-u(t))ds dxdt \\
    &\quad-a\int_s^T\xi(t)E^\gamma(t)\int_\Omega|u_t(t)|^{m(x)-2}u_t(t)u(t)dxdt\\
    &\quad+ b\int_s^T\xi(t)E^\gamma(t)\int_\Omega\Big(1-\frac{2}{p(x)}\Big)|u(t)|^{p(x)}dxdt,
\end{split}
\end{equation}
where we use the equality
\begin{equation*}
\begin{split}
&\int_s^T \xi(t)E^\gamma(t)\frac{d}{dt}\int_\Omega u(t)u_t(t)dxdt=-\int_{s}^{T}\frac{d}{d t}\Big[\xi(t)E^{\gamma}(t)\int_{\Omega}uu_{t}dx\Big]dt\\
&\quad+
\gamma\int_{s}^{T}\xi(t)E^{\gamma-1}(t)E'(t)\int_{\Omega}uu_{t}dxdt+
\int_{s}^{T}\xi'(t)E^{\gamma}(t)\int_{\Omega}uu_{t}dxdt.
\end{split}
\end{equation*}

In what follows, we estimate each term in the right-hand side of \eqref{3.9}.

\textbf{(1)} Estimate for $\Big|-\int_{s}^{T}\frac{d}{d t}\Big[\xi(t)E^{\gamma}(t)\int_{\Omega}uu_{t}dx\Big]dt\Big|$\\
Recalling \eqref{3.5}, we have
\begin{equation}\label{addequ1}
    \|\nabla u(t)\|_2^2\leq \frac{2(1+\tilde{C})}{l}E(t)\leq \frac{2(1+\tilde{C})}{l}E(s)\leq \frac{2(1+\tilde{C})}{l}E(0)\quad \forall t\ge s.
\end{equation}
It follows from Cauchy's inequality, $\omega_1\|u\|_2^2\leq \|\nabla u\|_2^2$, \eqref{3.5}, \eqref{addequ1}, \eqref{addequ1}  and \eqref{2.3} that
\begin{equation}\label{3.10}
\begin{split}
&\Big|-\int_{s}^{T}\frac{d}{d t}\Big[\xi(t)E^{\gamma}(t)\int_{\Omega}u(t)u_{t}(t)dx\Big]dt\Big|\\
&\quad\leq\Big|\xi(s)E^\gamma(s)\int_{\Omega}u(s)u_t(s)dx-\xi(T)E^\gamma(T)\int_{\Omega}u(T)u_t(T)dx\Big|\\
&\quad\leq \frac{\xi(s)E^\gamma(s)}{2}\Big[\|u(s)\|_2^2+\|u_t(s)\|_2^2
+\|u(T)\|_2^2+\|u_t(T)\|_2^2\Big]\\
&\quad\leq \xi(s)E^\gamma(s)\Big[\frac{1}{2\omega_1}\Big(\|\nabla u(s)\|_2^2+\|\nabla u(T)\|_2^2\Big)+\frac{1}{2}\Big(\|u_t(s)\|_2^2+\|u_t(T)\|_2^2\Big)\Big]\\
&
\quad\leq \Big[\frac{2(1+\tilde{C})}{\omega_1l}+2(1+\tilde{C})\Big]\xi(0)E^{\gamma}(0)E(s),
\end{split}
\end{equation}
here $\omega_1$ is the  first eigenvalue of $-\Delta$ operator under Dirichlet boundary.

\textbf{(2)} Estimate for $\Big|
\gamma\int_{s}^{T}\xi(t)E^{\gamma-1}(t)E'(t)\int_{\Omega}uu_{t}dxdt\Big|$\\
Similar to \eqref{3.10}, using again Cauchy's inequality, $\omega_1\|u\|_2^2\leq \|\nabla u\|_2^2$, \eqref{3.5}  and \eqref{2.3}, we get
\begin{equation}\label{3.11}
\begin{split}
&\Big|
\gamma\int_{s}^{T}\xi(t)E^{\gamma-1}(t)E'(t)\int_{\Omega}u(t)u_{t}(t)dxdt\Big|\\
&\quad=-\gamma\int_{s}^{T}\xi(t)E^{\gamma-1}(t)E'(t)\int_{\Omega}|u(t)||u_{t}(t)|dxdt\\
&\quad\leq-\frac{\gamma}{2}\int_{s}^{T}\xi(t)E^{\gamma-1}(t)E'(t)(\|u(t)\|_2^2+\|u_{t}(t)\|_2^{2})dt\\
&\quad\leq-\frac{\gamma }{2\omega_1}\int_{s}^{T}\xi(t)E^{\gamma-1}(t)E'(t)\|\nabla u(t)\|_2^2dt
-\frac{\gamma }{2}\int_{s}^{T}\xi(t)E^{\gamma-1}(t)E'(t)\|u_{t}(t)\|_2^{2}dt\\
&\quad\leq-\frac{\gamma(1+\tilde{C})}{\omega_1l}
\int_{s}^{T}\xi(t)E^{\gamma}(t)E'(t)dt
-\gamma(1+\tilde{C})\int_{s}^{T}\xi(t)E^{\gamma}(t)E'(t)dt\\
&\quad\leq-\frac{\gamma(1+\tilde{C})}{\omega_1l(\gamma+1)}
\int_{s}^{T}\frac{d}{dt}[\xi(t)E^{\gamma+1}(t)]dt+\frac{\gamma(1+\tilde{C})}{\omega_1l(\gamma+1)}\int_{s}^{T}\xi'(t)E^{\gamma+1}(t)dt\\
&\quad\quad
-\frac{\gamma(1+\tilde{C})}{\gamma+1}\int_{s}^{T}\frac{d}{dt}[\xi(t)E^{\gamma+1}(t)]dt+\frac{\gamma(1+\tilde{C})}{\gamma+1}\int_{s}^{T}\xi'(t)E^{\gamma+1}(t)dt\\
&\quad\leq\frac{\gamma(1+\tilde{C})}{\omega_1l(\gamma+1)}\Big[\xi(s)E^{\gamma+1}(s)-\xi(T)E^{\gamma+1}(T)\Big]
+\frac{\gamma(1+\tilde{C})}{\gamma+1}\Big[\xi(s)E^{\gamma+1}(s)-\xi(T)E^{\gamma+1}(T)\Big]\\
&\quad\leq\Big[\frac{\gamma(1+\tilde{C})}{\omega_1l(\gamma+1)}+\frac{\gamma(1+\tilde{C})}{\gamma+1}\Big]\xi(0)E^{\gamma}(0)E(s).
\end{split}
\end{equation}

\textbf{(3)} Estimate for $\Big|
\int_{s}^{T}\xi'(t)E^{\gamma}(t)\int_{\Omega}uu_{t}dxdt\Big|
$\\
Similar to \eqref{3.11}, one has
\begin{equation}\label{9add3.11}
\begin{split}
&\Big|
\int_{s}^{T}\xi'(t)E^{\gamma}(t)\int_{\Omega}uu_{t}dxdt\Big|\\
&\quad\leq-\frac{(1+\tilde{C})}{\omega_1l}
\int_{s}^{T}\xi'(t)E^{\gamma+1}(t)dt
-(1+\tilde{C})\int_{s}^{T}\xi'(t)E^{\gamma+1}(t)dt\\
&\quad\leq-\Big[\frac{(1+\tilde{C})}{\omega_1l}+(1+\tilde{C})\Big]
\Big[\xi(t)E^{\gamma+1}(t)|_s^T-(\gamma+1) \int_{s}^{T}\xi(t)E^{\gamma}(t)E'(t)dt\Big]\\
&\quad\leq\Big[\frac{(1+\tilde{C})}{\omega_1l}+(1+\tilde{C})\Big]\xi(0)E^{\gamma}(0)E(s).
\end{split}
\end{equation}

\textbf{(4)} Estimate for $2\int_{s}^{T}\xi(t)E^{\gamma}(t)\|u_{t}(t)\|_2^{2}dt$\\
To complete this estimate, we discuss by dividing  $m(x)$ into three cases.

\textsc{Case} 1: For $m_1>2$, using the continuous embedding $L^{m(x)}(\Omega)\hookrightarrow L^{2}(\Omega)$, \eqref{2.1}, Young's inequality with $0<\epsilon<1$ and \eqref{2.3}, then
\begin{equation}\label{3.12}
\begin{split}
&2\int_{s}^{T}\xi(t)E^{\gamma}(t)\|u_{t}(t)\|_2^{2}dt\\
&\quad\leq 2(1+|\Omega|)^2\int_{s}^{T}\xi(t)E^{\gamma}(t)\|u_t(t)\|_{m(x)}^2dt\\
&\quad\leq 2(1+|\Omega|)^2\int_{s}^{T}\xi(t)E^{\gamma}(t)\max\Big\{\Big(\int_\Omega |u_t(t)|^{m(x)}dx\Big)^{\frac{2}{m_1}},\Big(\int_\Omega |u_t(t)|^{m(x)}dx\Big)^{\frac{2}{m_2}}\Big\}dt\\
&\quad\leq 2(1+|\Omega|)^2\epsilon^{\frac{m_2}{m_2-2}}\int_{s}^{T}\xi(t)\max\Big\{E^{\frac{m_1\gamma}{m_1-2}}(t),E^{\frac{m_2\gamma}{m_2-2}}(t)\Big\}dt\\
&\quad\quad+
2(1+|\Omega|)^2\epsilon^{-\frac{m_2}{2}}\int_{s}^{T}\xi(t)\int_\Omega |u_t(t)|^{m(x)}dxdt\\
&\quad\leq 2(1+|\Omega|)^2\epsilon^{\frac{m_2}{m_2-2}}\max\Big\{E^{\frac{2(m_2-m_1)\gamma}{(m_1-2)(m_2-2)}}(0),1\Big\}\int_{s}^{T}\xi(t)E^{\frac{\gamma m_2}{m_2-2}}(t)dt\\
&\quad\quad+
2(1+|\Omega|)^2\epsilon^{-\frac{m_2}{2}}\frac1a\int_{s}^{T}\xi(t)(-E'(t))dxdt\\
&\quad\leq 2(1+|\Omega|)^2\epsilon^{\frac{m_2}{m_2-2}}\max\Big\{E^{\frac{m_2-m_1}{m_1-2}}(0),1\Big\}\int_{s}^{T}\xi(t)E^{\gamma+1}(t)dt
\\&\quad\quad+
2(1+|\Omega|)^2\epsilon^{-\frac{m_2}{2}}\frac1a\xi(0)E(s),
\end{split}
\end{equation}
here the following inequality holds due to \eqref{add3.2}
\[
\max\Big\{E^{\frac{m_1\gamma}{m_1-2}}(t),E^{\frac{m_2\gamma}{m_2-2}}(t)\Big\}\leq \max\Big\{E^{\frac{2(m_2-m_1)\gamma}{(m_1-2)(m_2-2)}}(0),1\Big\}E^{\frac{\gamma m_2}{m_2-2}},\]
and equality $\frac{m_2\gamma}{m_2-2}=\gamma+1$ illustrates $\gamma=\frac{m_2-2}{2}$.

\textsc{Case} 2: For $m_2>m_1=2$, we similarly get
\begin{equation}\label{3.13}
\begin{split}
&2\int_{s}^{T}E^{\gamma}(t)\xi(t)\|u_{t}(t)\|_2^{2}dt\\
&\quad\leq 2(1+|\Omega|)^2\int_{s}^{T}\xi(t)E^{\gamma}(t)\max\Big\{\int_\Omega |u_t(t)|^{m(x)}dx,\Big(\int_\Omega |u_t(t)|^{m(x)}dx\Big)^{\frac{2}{m_2}}\Big\}dt\\
&\quad\leq -2(1+|\Omega|)^2\frac1a\int_{s}^{T}\xi(t)E^{\gamma}(t)E'(t)dt+2(1+|\Omega|)^2\epsilon^{\frac{m_2}{m_2-2}}\int_{s}^{T}\xi(t)E^{\frac{m_2\gamma}{m_2-2}}(t)dt\\
&\quad\quad+
2(1+|\Omega|)^2\epsilon^{-\frac{m_2}{2}}\int_{s}^{T}\int_\Omega \xi(t)|u_t(t)|^{m(x)}dxdt\\
&\quad\leq -\frac{2(1+|\Omega|)^2}{\gamma+1}\frac1a\int_s^T\frac{d}{dt}[\xi(t)E^{\gamma+1}(t)]+\frac{2(1+|\Omega|)^2}{\gamma+1}\frac1a\int_s^T\xi'(t)E^{\gamma+1}(t)dt\\
&\quad\quad+2(1+|\Omega|)^2\epsilon^{\frac{m_2}{m_2-2}}\int_{s}^{T}\xi(t)E^{\gamma+1}(t)dt+
2(1+|\Omega|)^2\epsilon^{-\frac{m_2}{2}}\xi(0)E(s)\\
&\quad\leq \frac{2(1+|\Omega|)^2}{\gamma+1}\frac1a\xi(0)E^\gamma(0)E(s)+2(1+|\Omega|)^2\epsilon^{\frac{m_2}{m_2-2}}\int_{s}^{T}\xi(t)E^{\gamma+1}(t)dt\\
&\quad\quad+
2(1+|\Omega|)^2\epsilon^{-\frac{m_2}{2}}\xi(0)E(s).
\end{split}
\end{equation}

\textsc{Case} 3: For $m_2=m_1=2$, it directly follows that
\begin{equation}\label{add3.3}
\begin{split}
    2\int_{s}^{T}\xi(t)E^{\gamma}(t)\|u_{t}(t)\|_2^{2}dt&\leq -\frac2a\int_{s}^{T}\xi(t)E^{\gamma}(t)E'(t)dt\\
    &\quad\leq \frac{2}{(\gamma+1)a}\xi(0)E^\gamma(0)E(s).
    \end{split}
\end{equation}

\textbf{(5)} Estimate for $\int_s^T\xi(t)E^\gamma(t)\int_0^tg(t-s)\|\nabla u(t)-\nabla u(s)\|_2^2dsdt$\\
We have different estimates when the relaxation function $g$ satisfies specific conditions.

\textsc{Type} \uppercase\expandafter{\romannumeral1}: $g'(t)\le -\xi(t)g(t)$.\\
Combining \eqref{2.3} with \eqref{3.5} is to obtain
\begin{equation}\label{3.14}
\begin{split}
&\int_s^T\xi(t)E^\gamma(t)\int_0^tg(t-s)\|\nabla u(t)-\nabla u(s)\|_2^2dsdt\\
&\quad\leq -\int_s^TE^\gamma(t)\int_0^tg_t(t-s)\|\nabla u(t)-\nabla u(s)\|_2^2dsdt\\
&\quad\leq 2\int_s^TE^\gamma(t)(-E'(t))dt\\
&\quad\leq \frac{2}{\gamma+1}E^{\gamma}(0)E(s).
\end{split}
\end{equation}

\textsc{Type} \uppercase\expandafter{\romannumeral2}: $g'(t)+Cg^\alpha(t)\leq 0$, where $1 <\alpha< 2.$ \\
Let us borrow the part ideas from \cite{CO2003} to deal with this type. Note that one may choose $\xi(t)=1$ in this type. Since $g(t)> 0$, then $g(t)$ decays polynomially fast, that is, $0 < g(t)\leq C(1 + t)^{-\frac{1}{\alpha-1}}$, where $\frac{1}{\alpha-1}\in(1, \infty)$. It directly follows from H\"{o}lder's inequality and Young's inequality that for $0<\sigma<1$,
\begin{equation}\label{add3.4}
\begin{split}
&\int_s^TE^\gamma(t)\int_0^tg(t-s)\|\nabla u(t)-\nabla u(s)\|_2^2dsdt\\
&\quad=\int_s^TE^\gamma(t)\int_0^t[g(t-s)]^{\frac{(1-\sigma)(\alpha-1)}{\sigma+\alpha-1}}\|\nabla u(t)-\nabla u(s)\|_2^{\frac{2(\alpha-1)}{\sigma+\alpha-1}}\\
&\quad\quad\times[g(t-s)]^{\frac{\sigma\alpha}{\sigma+\alpha-1}}\|\nabla u(t)-\nabla u(s)\|_2^{\frac{2\sigma}{\sigma+\alpha-1}}dsdt\\
&\quad\leq\int_s^TE^\gamma(t)\Big(\int_0^t[g(t-s)]^{1-\sigma}\|\nabla u(t)-\nabla u(s)\|_2^2ds\Big)^{\frac{\alpha-1}{\sigma+\alpha-1}}\\
&\quad\quad\times\Big(\int_0^t[g(t-s)]^\alpha\|\nabla u(t)-\nabla u(s)\|_2^2ds\Big)^{\frac{\sigma}{\sigma+\alpha-1}}dt\\
&\quad\leq\frac{\alpha-1}{\sigma+\alpha-1}\int_s^T\int_0^t[g(t-s)]^{1-\sigma}\|\nabla u(t)-\nabla u(s)\|_2^2dsdt\\
&\quad\quad+\frac{\sigma}{\sigma+\alpha-1}\int_s^TE^\frac{\gamma(\sigma+\alpha-1)}{\sigma}(t)\int_0^t[g(t-s)]^\alpha\|\nabla u(t)-\nabla u(s)\|_2^2dsdt.
\end{split}
\end{equation}
Thus, combining \eqref{addequ1} and $0 < g(t)\leq C(1 + t)^{-\frac{1}{\alpha-1}}$ is to get
\begin{equation}\label{add3.5}
\begin{split}
    &\frac{\alpha-1}{\sigma+\alpha-1}\int_s^T\int_0^t[g(t-s)]^{1-\sigma}\|\nabla u(t)-\nabla u(s)\|_2^2dsdt\\
    &\quad\leq \frac{\alpha-1}{\sigma+\alpha-1}\frac{4(1+\tilde{C})}{l}E(s)\int_s^T\int_0^t[g(t-s)]^{1-\sigma}dsdt\\
    &\quad= \frac{\alpha-1}{\sigma+\alpha-1}\frac{4(1+\tilde{C})}{l}E(s)\int_s^T\frac{C^{1-\sigma}}{1-\frac{1-\sigma}{\alpha-1}}(1+t)^{1-\frac{1-\sigma}{\alpha-1}}dt\\
    &\quad\leq \frac{\alpha-1}{\sigma+\alpha-1}\frac{4(1+\tilde{C})}{l}E(s)
\frac{C^{1-\sigma}}{1-\frac{1-\sigma}{\alpha-1}}\frac{1}{2-\frac{1-\sigma}{\alpha-1}}(1+T)^{2-\frac{1-\sigma}{\alpha-1}}\\
&\quad\leq \frac{\alpha-1}{\sigma+\alpha-1}\frac{4(1+\tilde{C})}{l}E(s)
\frac{C^{1-\sigma}}{1-\frac{1-\sigma}{\alpha-1}}\frac{1}{2-\frac{1-\sigma}{\alpha-1}}
\end{split}
\end{equation}
by assuming $2<\frac{1-\sigma}{\alpha-1}$, i.e. $2\alpha+\sigma<3$. Using $g'(t)+Cg^\alpha(t)\leq 0$ and \eqref{2.3}, then
\begin{equation}\label{add3.6}
    \begin{split}
    &\frac{\sigma}{\sigma+\alpha-1}\int_s^TE^\frac{\gamma(\sigma+\alpha-1)}{\sigma}(t)\int_0^t[g(t-s)]^\alpha\|\nabla u(t)-\nabla u(s)\|_2^2dsdt\\
    &\quad \leq - \frac{\sigma}{C(\sigma+\alpha-1)}\int_s^TE^\frac{\gamma(\sigma+\alpha-1)}{\sigma}(t)\int_0^tg_t(t-s)\|\nabla u(t)-\nabla u(s)\|_2^2dsdt\\
     &\quad \leq  \frac{\sigma}{C(\sigma+\alpha-1)}\int_s^TE^\frac{\gamma(\sigma+\alpha-1)}{\sigma}(t)(-E'(t))dt\\
     &\quad=\frac{\sigma}{C(\sigma+\alpha-1)}\frac{1}{\frac{\gamma(\sigma+\alpha-1)}{\sigma}+1}
     E^{\frac{\gamma(\sigma+\alpha-1)}{\sigma}}(0)E(s).
    \end{split}
\end{equation}
Thus, inserting \eqref{add3.5} and \eqref{add3.6} into \eqref{add3.4}, one obtains
\begin{equation}\label{add3.7}
\begin{split}
&\int_s^TE^\gamma(t)\int_0^tg(t-s)\|\nabla u(t)-\nabla u(s)\|_2^2dsdt\\
&\quad=\frac{\alpha-1}{\sigma+\alpha-1}\frac{4(1+\tilde{C})}{l}E(s)
\frac{C^{1-\sigma}}{1-\frac{1-\sigma}{\alpha-1}}\frac{1}{2-\frac{1-\sigma}{\alpha-1}}\\
&\quad\quad+\frac{\sigma}{C(\sigma+\alpha-1)}\frac{1}{\frac{\gamma(\sigma+\alpha-1)}{\sigma}+1}
     E^{\frac{\gamma(\sigma+\alpha-1)}{\sigma}}(0)E(s).
\end{split}
\end{equation}

\textbf{(6)} Estimate for $\Big|\int_s^T\xi(t)E^\gamma(t)\int_\Omega\int_0^t g(t-s)\nabla  u(t)\nabla (u(s)-u(t))ds dxdt\Big|$

\textsc{Type} \uppercase\expandafter{\romannumeral1}: $g'(t)\le-\xi(t)g(t)$.\\
Using Cauchy's inequality with $0<\varepsilon<1$, \eqref{addequ1} \eqref{2.1} and \eqref{3.14}, we have
\begin{equation}\label{3.15}
\begin{split}
  &\Big|\int_s^T \xi(t)E^\gamma(t)\int_\Omega\int_0^t g(t-s)\nabla  u(t)\nabla (u(s)-u(t))dsdt dx\Big|\\
   &\quad\leq\frac\varepsilon2\int_s^T\xi(t)E^\gamma(t)\int_0^t g(t-s)\|\nabla  u(t)\|_2^2dsdt \\
   &\quad\quad+\frac{1}{2\varepsilon}\int_s^T\xi(t)E^\gamma(t)\int_0^t g(t-s)\|\nabla (u(s)-u(t))\|_2^2dsdt\\
   &\quad\leq\varepsilon(1+\tilde{C})\frac{1-l}{l}\int_s^T\xi(t)E^{\gamma+1}(t)dt +\frac{1}{2\varepsilon}\frac{2}{\gamma+1}E^{\gamma}(0)E(s).
\end{split}
\end{equation}

\textsc{Type} \uppercase\expandafter{\romannumeral2}: $g'(t)+Cg^\alpha(t)\leq 0$, where $1 <\alpha< 2.$ \\
Similarly, using \eqref{add3.7}, it follows that 
\begin{equation}\label{addeq1}
\begin{split}
  &\Big|\int_s^T E^\gamma(t)\int_\Omega\int_0^t g(t-s)\nabla  u(t)\nabla (u(s)-u(t))dsdt dx\Big|\\
   &\quad\leq\varepsilon(1+\tilde{C})\frac{1-l}{l}\int_s^TE^{\gamma+1}(t)dt \\ &\quad\quad+\frac{1}{2\varepsilon}\frac{\alpha-1}{\sigma+\alpha-1}\frac{4(1+\tilde{C})}{l}E(s)
\frac{C^{1-\sigma}}{1-\frac{1-\sigma}{\alpha-1}}\frac{1}{2-\frac{1-\sigma}{\alpha-1}}\\
&\quad\quad+\frac{1}{2\varepsilon}\frac{\sigma}{C(\sigma+\alpha-1)}\frac{1}{\frac{\gamma(\sigma+\alpha-1)}{\sigma}+1}
     E^{\frac{\gamma(\sigma+\alpha-1)}{\sigma}}(0)E(s).
\end{split}
\end{equation}

\textbf{(7)} Estimate for $\Big|-a\int_s^T\xi(t)E^\gamma(t)\int_\Omega|u_t(t)|^{m(x)-2}u_t(t)u(t)dxdt\Big|$\\
It follows by using $\hbox{Young's}$ inequality with $0<\delta<1$ that
\begin{equation}
\label{3.16}
\begin{split}
&\Big|-a\int_s^T\xi(t)E^\gamma(t)\int_\Omega|u_t(t)|^{m(x)-2}u_t(t)u(t)dxdt\Big|\\
&\quad\leq a\int_{s}^{T} \xi(t)E^\gamma(t) \Big[\int_{\Omega}\Big(\delta^{-\frac{m_1}{m_1-1}}|u_{t}|^{m(x)}+\delta^{m_1}|u|^{m(x)}\Big)dx\Big]dt
\\
&\quad\leq-\delta^{-\frac{m_1}{m_1-1}}\int_{s}^{T} \xi(t)E^\gamma(t)  E'(t)dt+\delta^{m_1}a\int_{s}^{T}  \xi(t)E^\gamma(t) \int_{\Omega}|u|^{m(x)}dxdt\\
&\quad\leq\frac{\delta^{-\frac{m_1}{m_1-1}}}{\gamma+1}\xi(0)E^{\gamma}(0)E(s)+\delta^{m_1}a\frac{2B_1^{m_2}}{l^{\frac{m_2}{2}}}(1+\tilde{C}) \int_{s}^{T}\xi(t)E^{\gamma+1}(t)dt,
\end{split}
\end{equation}
here, we apply
\begin{equation*}
\begin{split}
\int_{\Omega}|u|^{m(x)}dx&\leq \max\Big\{\frac{B_1^{m_1}}{l^{\frac{m_1}{2}}}\Big(l^{\frac12}\|\nabla u(t)\|_2\Big)^{m_1},\frac{B_1^{m_2}}{l^{\frac{m_2}{2}}}\Big(l^{\frac12}\|\nabla u(t)\|_2\Big)^{m_2}\Big\}\\
&\quad\leq \frac{B_1^{m_2}}{l^{\frac{m_2}{2}}}\Big(l^{\frac12}\|\nabla u(t)\|_2\Big)^{2}\leq \frac{2B_1^{m_2}}{l^{\frac{m_2}{2}}}(1+\tilde{C})E(t)
\end{split}
\end{equation*}
by Lemma \ref{lem3.1} and \eqref{addequ1}.

\textbf{(8)} Estimate for $b\int_s^T\xi(t)E^\gamma(t)\int_\Omega\Big(1-\frac{2}{p(x)}\Big)|u(t)|^{p(x)}dxdt$\\
It directly follows from \eqref{3.3} that
\begin{equation}\label{3.17}
\begin{split}
  b\int_s^T\xi(t)E^\gamma(t)\int_\Omega\Big(1-\frac{2}{p(x)}\Big)|u(t)|^{p(x)}dxdt\leq \Big(1-\frac{2}{p_2}\Big)\tilde{C}p_2\int_s^T\xi(t)E^{\gamma+1}(t)dt.
\end{split}
\end{equation}

We are now in a position to finish the estimate for \eqref{3.9}. Let us show the process by setting \textsc{Case} 1($m_1>2$) as an example.

\textsc{Type} \uppercase\expandafter{\romannumeral1}: $g'(t)+\le-\xi(t)g(t)$.\\
Combining \eqref{3.9} with \eqref{3.10}-\eqref{3.12} as well as \eqref{3.14},\eqref{3.15},\eqref{3.16}, \eqref{3.17}, we obtain
\begin{equation}\label{3.18}
\begin{split}
2\int_s^T E^{\gamma+1}(t)dt
&\leq\Bigg\{2(1+|\Omega|)^2\epsilon^{\frac{m_2}{m_2-2}}\max\Big\{E^{\frac{m_2-m_1}{m_1-2}}(0),1\Big\}\\
&\quad\quad+\varepsilon(1+\tilde{C})\frac{1-l}{l}+\delta^{m_1}a\frac{2B_1^{m_2}}{l^{\frac{m_2}{2}}}(1+\tilde{C}) \Bigg\}\int_{s}^{T}E^{\gamma+1}(t)dt\\
&\quad\quad+\Bigg\{\Big[\frac{3(1+\tilde{C})}{\omega_1l}+3(1+\tilde{C})\Big]+
\Big[\frac{\gamma(1+\tilde{C})}{\omega_1l(\gamma+1)}+\frac{\gamma(1+\tilde{C})}{\gamma+1}\Big]\\
&\quad\quad+ 2(1+|\Omega|)^2\epsilon^{-\frac{m_2}{2}}\frac{1}{aE^{\gamma}(0)}
+\Big(1+\frac{1}{2\varepsilon}\Big)\frac{2}{\xi(0)(\gamma+1)}\\
&\quad\quad+\frac{\delta^{-\frac{m_1}{m_1-1}}}{\gamma+1}\Bigg\}\xi(0)E^{\gamma}(0)E(s)
+\Big(1-\frac{2}{p_2}\Big)\tilde{C}p_2\int_s^T\xi(t)E^{\gamma+1}(t)dt.
\end{split}
\end{equation}
The condition \[0<E(0)=f(\lambda_2)<f\Big(\Big(\frac{p_1}{p_2}\Big)^{\frac{1}{p_1-2}}\lambda_1\Big)=\Big(\frac{p_1}{p_2}\Big)^{\frac{1}{p_1-2}}\lambda_1^2\Big(\frac12-\frac{1}{p_2}\Big)\]
 and the monotonicity of $f(\lambda)$ easily imply
\[\lambda_2<\Big(\frac{p_1}{p_2}\Big)^{\frac{1}{p_1-2}}\lambda_1\leq\lambda_1<1,\]
which further illustrates
\[\omega:=\Big(1-\frac{2}{p_2}\Big)\tilde{C}p_2<2.\]
Let us choose $0<\epsilon,~\varepsilon,~\delta<1$ sufficiently small such that
\[2(1+|\Omega|)^2\epsilon^{\frac{m_2}{m_2-2}}\max\Big\{E^{\frac{m_2-m_1}{m_1-2}}(0),1\Big\}+\varepsilon(1+\tilde{C})\frac{1-l}{l}+\delta^{m_1}a\frac{2B_1^{m_2}}{l^{\frac{m_2}{2}}}(1+\tilde{C}) =\frac{2-\omega}{2}.\]
Therefore, \eqref{3.18} can be rewritten as
\[\int_s^T \xi(t)E^{\gamma+1}(t)dt\leq\frac {1}{K}E^{\gamma}(0)E(s),\]
 where
 \begin{equation}\label{add3.8}
 \begin{split}
 K&=\Bigg\{\Big[\frac{3(1+\tilde{C})}{\omega_1l}+3(1+\tilde{C})\Big]+
\Big[\frac{\gamma(1+\tilde{C})}{\omega_1l(\gamma+1)}+\frac{\gamma(1+\tilde{C})}{\gamma+1}\Big]\\
&\quad\quad+ 2(1+|\Omega|)^2\epsilon^{-\frac{m_2}{2}}\frac{1}{aE^{\gamma}(0)}
+\Big(1+\frac{1}{2\varepsilon}\Big)\frac{2}{\xi(0)(\gamma+1)}+\frac{\delta^{-\frac{m_1}{m_1-1}}}{\gamma+1}\Bigg\}^{-1}\frac{1}{\xi(0)}\frac{2-\omega}{2}.
\end{split}
\end{equation}
Obviously,
\begin{equation}
\label{3.19}
\int_s^{+\infty} \xi(t)E^{\gamma+1}(t)dt\leq\frac {1}{K}E^{\gamma}(0)E(s)
\end{equation}
by letting $T\to +\infty$. Choosing $\phi(t)=\int_0^t\xi(s)ds$ in Lemma \ref{lem3.0}, we
directly obtain  \eqref{00add0}.

\textsc{Type} \uppercase\expandafter{\romannumeral2}: $g'(t)+Cg^\alpha(t)\leq 0$.\\
Combining \eqref{3.9} and \eqref{3.10}-\eqref{3.12} as well as \eqref{add3.7}, \eqref{addeq1}-\eqref{3.17},  and letting $\xi(t)=1$, similarly, we get
\begin{equation}\label{3.20}
\int_s^{+\infty} E^{\gamma+1}(t)dt\leq\frac {1}{K(\alpha,\sigma)}E^{\gamma}(0)E(s)
\end{equation}
with
\begin{equation}\label{3.21}
 \begin{split}
 K(\alpha,\sigma)&=\Bigg\{\Big[\frac{2(1+\tilde{C})}{\omega_1l}+2(1+\tilde{C})\Big]+
\Big[\frac{\gamma(1+\tilde{C})}{\omega_1l(\gamma+1)}+\frac{\gamma(1+\tilde{C})}{\gamma+1}\Big]\\
&\quad\quad+ 2(1+|\Omega|)^2\epsilon^{-\frac{m_2}{2}}\frac{1}{aE^{\gamma}(0)}+\frac{2}{C(\gamma+1)}\\
&\quad\quad+\Big(1+\frac{1}{2\varepsilon}\Big)\Big[\frac{\alpha-1}{\sigma+\alpha-1}\frac{4(1+\tilde{C})}{l}\frac{1}{E^{\gamma}(0)}
\frac{C^{1-\sigma}}{1-\frac{1-\sigma}{\alpha-1}}\frac{1}{2-\frac{1-\sigma}{\alpha-1}}\\
&\quad\quad+\frac{\sigma}{C(\sigma+\alpha-1)}\frac{1}{\frac{\gamma(\sigma+\alpha-1)}{\sigma}+1}
     E^{\frac{\gamma(\sigma+\alpha-1)}{\sigma}-\gamma}(0)\Big]+\frac{\delta^{-\frac{m_1}{m_1-1}}}{\gamma+1}\Bigg\}^{-1}\frac{2-\omega}{2}.
\end{split}
\end{equation}
Lemma \ref{lem3.0} directly illustrates \eqref{00add1}.

For \textsc{Case} 2($m_2>m_1=2$), similar to \textsc{Case} 1, we still prove \eqref{00add0} and \eqref{00add1}. In fact, we easily find the corresponding $K$ and $K(\alpha,\sigma)$ like $\eqref{3.19}$ and \eqref{3.21} with a slight difference, respectively.

For  \textsc{Case} 3($m_2=m_1=2$), i.e. $m(x)=2$, let us set $\gamma=0$, then similar to \textsc{Case} 1, we can prove  \eqref{add0} and \eqref{add1}.
\end{proof3.1}

\subsection*{Acknowledgements}
The first author wishes to express her sincere gratitude to Professor Wenjie Gao  for his support and constant encouragement.

\end{document}